\documentclass[11pt,a4paper,reqno]{amsart}
\usepackage{hyperref}
\usepackage[all]{xy}
\usepackage{amssymb}
\usepackage{amsmath}
\usepackage{mathtools}

\usepackage{color}
\usepackage[usernames,dvipsnames]{xcolor}

\font\cyr=wncyr10 scaled \magstep1%
\def\Sh{\text{\cyr Sh}}
\newcommand{\xrightarrowdbl}[2][]{%
  \xrightarrow[#1]{#2}\mathrel{\mkern-14mu}\rightarrow
}

\DeclareMathOperator{\cok}{coker}

\newcommand{\fp}{\mathfrak{p}}

\newcommand{\fg}{\mathfrak{g}}
\newcommand{\Hom}{\text{Hom}}
\newcommand{\End}{\text{End}}

\newcommand{\gen}{\text{gen}}

\newcommand{\Sp}{\text{Spec} \,}

\chardef\bslash=`\\ 

\newtheorem{theorem}{Theorem}[section]

\newtheorem{prop}[theorem]{Proposition}
\newtheorem{lem}[theorem]{Lemma}
\newtheorem{cor}[theorem]{Corollary}

\theoremstyle{definition}

\newtheorem{remark}[theorem]{Remark}
\newtheorem{example}[theorem]{Example}

\numberwithin{equation}{section}

\newtheorem*{maintheorem*}{Main Theorem}
\theoremstyle{definition}
\newtheorem{definition}{Definition}

\newtheorem{notation}[theorem]{Notation}

\newcommand{\surj}{\twoheadrightarrow}

\newcommand{\CO}{\mathcal{O}}
\newcommand{\CS}{\mathcal{O}_S}
\newcommand{\Oiy}{\mathcal{O}_{\{\iy\}}}
\newcommand{\ClS}{\text{Cl}_S}

\newcommand{\SOV}{\un{\textbf{SO}}_q}

\newcommand{\wh}{\widehat}

\renewcommand{\sectionmark}[1]{}
\renewcommand{\Im}{\operatorname{Im}}

\newcommand{\af}{\text{af}}
\newcommand{\Br}{{\mathrm{Br}}}

\newcommand{\iy}{\infty}

\newcommand{\bk}{\bigskip}
\newcommand{\fc}{\frac}
\newcommand{\G}{\Gamma}

\renewcommand{\sc}{\text{sc}}

\newcommand{\Pic}{\text{Pic~}}

\newcommand{\dl}{\delta}

\newcommand{\lm}{\lambda}

\newcommand{\Om}{\Omega}
\newcommand{\ov}{\overline}
\newcommand{\un}{\underline}

\newcommand{\BG}{\mathbb{G}}
\newcommand{\BF}{\mathbb{F}}

\newcommand{\Q}{\mathbb{Q}}
\newcommand{\Z}{\mathbb{Z}}

\renewcommand{\a}{\alpha}

\newcommand{\et}{\text{\'et}}

\newcommand{\p}{\varphi}

\newcommand{\A}{\mathbb{A}}

\begin{document}



\baselineskip 20pt
\setcounter{equation}{0}
\pagestyle{plain}
\pagenumbering{arabic}

\title{On the genera of semisimple groups defined over an integral domain of a global function field}
\author{Rony A. Bitan}
\thanks{The research  was partially supported  by the ERC grant  291612.}

\begin{abstract}
Let $K=\BF_q(C)$ be the global function field of rational functions over a smooth and projective curve $C$ defined over a finite field $\BF_q$.  
The ring of regular functions on $C-S$ where $S \neq \emptyset$ is any finite set of closed points on $C$  is a Dedekind domain $\CS$ of $K$. 
For a semisimple $\CS$-group $\un{G}$ with a smooth fundamental group $\un{F}$, 
we aim to describe both the set of genera of $\un{G}$ and its principal genus (the latter if $\un{G} \otimes_{\CS} K$ is isotropic at $S$) 
in terms of abelian groups depending on $\CS$ and $\un{F}$ only. 
This leads to a necessary and sufficient condition for the Hasse local-global principle to hold for certain $\un{G}$. 
We also use it to express the Tamagawa number $\tau(G)$ of a semisimple $K$-group $G$ by the Euler Poincar\'e invariant.    
This facilitates the computation of $\tau(G)$ for twisted $K$-groups.     
\end{abstract}

\maketitle
\markright{On the genera of groups over an integral domain of a global function field}  

\pagestyle{headings}

\section{Introduction} \label{Introduction}
Let $C$ be a projective algebraic curve defined over a finite field $\BF_q$, 
assumed to be geometrically connected and smooth. 
Let $K=\BF_q(C)$ be the global field of rational functions over $C$, 
and let $\Om$ be the set of all closed points of $C$. 
For any point $\fp \in \Om$, let $v_\fp$ be the induced discrete valuation on $K$,    
$\hat{\CO}_\fp$ the complete valuation ring with respect to $v_\fp$,    
and $\hat{K}_\fp, k_\fp$ its fraction field and residue field at $\fp$, respectively.  
Any \emph{Hasse set} of $K$, namely, a non-empty finite set $S \subset \Om$,  
gives rise to an integral domain of $K$ called a \emph{Hasse domain}: 
$$ \CS := \{x \in K: v_\fp(x) \geq 0 \ \forall \fp \notin S\}. $$ 
This is a regular and one dimensional Dedekind domain. 
Group schemes defined over $\Sp \CS$ are underlined, being omitted in the notation of their generic fibers.  

\bk
  
Let $\un{G}$ be an affine, smooth and of finite type group scheme defined over $\Sp \CS$. 
We define $H^1_\et(\CS,\un{G})$ to be the set of isomorphism classes
of $\un{G}$-torsors over $\Sp \CS$ relative to the \'etale or the flat topology 
(the classification for the two topologies coincide when $\un{G}$ is smooth; cf. \cite[VIII~Cor.~2.3]{SGA4}).  
The sets $H^1(K,G)$ and $H^1_\et(\hat{\CO}_\fp,\un{G}_\fp)$, for every $\fp \notin S$, are defined similarly. 
All these three sets are naturally pointed:  
the distinguished point of $H^1_\et(\CS,\un{G})$ (resp., $H^1(K,G)$, $H^1_\et(\hat{\CO}_\fp,\un{G}_\fp)$) 
is the class of the trivial $\un{G}$-torsor $\un{G}$ (resp. trivial $G$-torsor $G$, trivial $\un{G}_\fp$-torsor $\un{G}_\fp$). 
There exists a canonical map of pointed-sets (mapping the distinguished point to the distinguished point): 
\begin{equation} \label{lm}
\lm: H^1_\et(\CS,\un{G}) \to H^1(K,G) \times \prod\limits_{\fp \notin S} H^1_\et(\hat{\CO}_\fp,\un{G}_\fp)       
\end{equation}
which is defined by mapping a class in $H^1_\et(\CS,\un{G})$ represented by $X$  
to the class 
represented by $(X \otimes_{\CS} \Sp K) \times \prod_{\fp \notin S} X \otimes_{\CS} \Sp \hat{\CO}_\fp$.  
Let $[\xi_0] := \lm([\un{G}])$. 
The \emph{principal genus} of $\un{G}$ is then $\lm^{-1}([\xi_0])$, 
i.e., the classes of $\un{G}$-torsors over $\Sp \CS$ that are generically and locally trivial at all points of $\CS$. 
More generally, a \emph{genus} of $\un{G}$ is any fiber $\lm^{-1}([\xi])$ where $[\xi] \in \Im(\lm)$.  
The \emph{set of genera} of $\un{G}$ is then: 
$$ \text{gen}(\un{G}) := \{ \lm^{-1}([\xi]) \ : \ [\xi] \in \Im(\lm) \}, $$ 
whence $H^1_\et(\CS,\un{G})$ is a disjoint union of its genera. 

\bk

Given a representative $P$ of a class in $H^1_\et(\CS,\un{G})$, 
by referring also to $\un{G}$ as a $\un{G}$-torsor acting on itself by conjugations, 
the quotient of $P \times_{\CS} \un{G}$
by the $\un{G}$-action $(p,g) \mapsto (ps^{-1},sgs^{-1})$   
is an affine $\CS$-group scheme ${^P}\un{G}$, 
called the \emph{twist} of $\un{G}$ by $P$.  
It is an inner form of $\un{G}$, thus is locally isomorphic to $\un{G}$ in the \'etale topology, 
namely, every fiber of it at a prime of $\CS$ is isomorphic to $\un{G}_\fp := \un{G} \otimes_{\CS} \hat{\CO}_\fp$ 
over some finite \'etale extension of $\hat{\CO}_\fp$. 
The map $\un{G} \mapsto {^P}\un{G}$ defines a bijection of pointed-sets $H^1_\et(\CS,\un{G}) \to H^1_\et(\CS,{^P}\un{G})$ (e.g., \cite[\S 2.2, Lemma 2.2.3, Examples 1,2]{Sko}). 

\bk

A group scheme defined over $\Sp \CS$ is said to be \emph{reductive} if  
it is affine and smooth over $\Sp \CS$, and each geometric fiber of it at a prime $\fp$ is (connected) reductive over $k_\fp$ (\cite[Exp.~XIX Def.~2.7]{SGA3}).  
It is \emph{semisimple} if it is reductive, and the rank of its root system equals that of its lattice of weights (\cite[Exp.~XXI Def.~1.1.1]{SGA3}).  
Suppose $\un{G}$ is semisimple and that its fundamental group $\un{F}$ is of order prime to $\text{char}(K)$. 
Being finite, of multiplicative type (\cite[XXII, Cor.~4.1.7]{SGA3}), commutative and smooth,   
$\un{F}$ decomposes into finitely many factors of the form $\text{Res}_{R/\CS}(\un{\mu}_m)$ 
or $\text{Res}^{(1)}_{R/\CS}(\un{\mu}_m)$ where $\un{\mu}_{m} := \Sp \CS[t]/(t^{m}-1)$ 
and $R$ is some finite (possibly trivial) \'etale extension of $\CS$. 
Consequently, $H^r_\et(\CS,\un{F})$ are abelian groups for all $r \geq 0$. 
The following two $\CS$-invariants of $\un{F}$ will play a major role in the description of $H^1_\et(\CS,\un{G})$:  

\begin{definition} \label{i}
Let $R$ be a finite \'etale extension of $\CS$. 
We define: 
\begin{align*} 
i(\un{F}) := \left \{ \begin{array}{l l}
\Br(R)[m]                        &  \un{F} = \text{Res}_{R/\CS}(\un{\mu}_m) \\
\ker(\Br(R)[m] \xrightarrow{N^{(2)}} \Br(\CS)[m])  &  \un{F} = \text{Res}^{(1)}_{R/\CS}(\un{\mu}_m) 
\end{array}\right.  
\end{align*} 
where $N^{(2)}$ is induced by the norm map $N_{R/\CS}$ 
and for a group $*$, $*[m]$ stands for its $m$-torsion~part.  
For $\un{F} = \prod_{k=1}^r \un{F}_k$ where each $\un{F}_k$ is one of the above, 
$i(\un{F})$ is the direct product $\prod_{k=1}^r i(\un{F}_k)$.

We also define for such $R$: 
\begin{align}  
j(\un{F}) := \left \{ \begin{array}{l l}
\Pic(R)/m                                                          &  \un{F} = \text{Res}_{R/\CS}(\un{\mu}_m) \\
\ker \left( \Pic(R)/m \xrightarrow{N^{(1)}/m} \Pic(\CS)/m \right)  &  \un{F} = \text{Res}^{(1)}_{R/\CS}(\un{\mu}_m) \\
\end{array}\right.  
\end{align} 
where $N^{(1)}$ is induced by $N_{R/\CS}$, and again $j(\prod_{k=1}^r \un{F}_k) := \prod_{k=1}^r j(\un{F}_k)$. 
\end{definition}

\begin{definition} \label{admissible}
We call $\un{F}$ \emph{admissible} if it is a finite direct product of the following factors: 
\begin{itemize}
\item[(1)] $\text{Res}_{R/\CS}(\un{\mu}_m)$, 
\item[(2)] $\text{Res}^{(1)}_{R/\CS}(\un{\mu}_m), [R:\CS]$ is prime to $m$,
\end{itemize} 
where $R$ is any finite \'etale extension of $\CS$. 
\end{definition}
 
After computing in Section \ref{Section: class set} the cohomology sets of some related $\CS$-groups, 
we observe in Section~\ref{Section genera} Proposition \ref{sequence of wG}, 
that if $\un{F}$ is admissible then there exists an exact sequence of pointed~sets:  
$$ 1 \to \ClS(\un{G}) \xhookrightarrow{h} H^1_\et(\CS,\un{G}) \xrightarrow{w_{\un{G}}} i(\un{F}) \to 1.  $$
We deduce in Corollary \ref{genera} that $\text{gen}(\un{G})$ bijects to $i(\un{F})$. 
In Section \ref{Section genus}, Theorem \ref{genus isotropic}, we show that $\ClS(\un{G})$ surjects onto $j(\un{F})$. 
If $G_S := \prod_{s \in S} G(\hat{K}_s)$ is non-compact, 
then this is a bijection.  
This leads us to formulate in Corollary~\ref{criterion} a necessary 
and sufficient condition for the \emph{Hasse local-global principle} to hold for $\un{G}$.  
In Section \ref{Section application}, we use the above results to express in Theorem \ref{tau G 2} 
the Tamagawa number $\tau(G)$ of an almost simple $K$-group $G$ with an admissible fundamental group~$F$,    
using the (restricted) Euler-Poincar\'e characterstic of some $\CS$-model of $F$ 
and a local invariant, and show how this new description 
facilitates the computation of $\tau(G)$ when $G$ is a twisted group.  

\bk

\section{\'Etale cohomology} \label{Section: class set}
\subsection{The class set}
Consider the ring of $S$-integral ad\`eles $\A_S := \prod_{\fp \in S} \hat{K}_\fp \times \prod_{\fp \notin S} \hat{\CO}_\fp$, 
being a subring of the ad\`eles $\A$.     
The $S$-\emph{class set} of an affine and of finite type $\CS$-group $\un{G}$ is the set of double cosets:  
$$ \ClS(\un{G}) := \un{G}(\A_S) \backslash \un{G}(\A) / G(K)  $$
(when over each $\hat{\CO}_\fp$ the above local model $\un{G}_\fp$ is taken). 
It is finite (cf. \cite[Proposition~3.9]{BP}), and its cardinality, 
called the $S$-\emph{class number} of $\un{G}$, is denoted by $h_S(\un{G})$. 
According to Nisnevich (\cite[Thm.~I.3.5]{Nis}) if $\un{G}$ is smooth, the map $\lm$ introduced in \eqref{lm} applied to it  
forms the following exact sequence of pointed-sets 
(when the trivial coset is considered as the distinguished point in $\ClS(\un{G})$): 
\begin{equation} \label{Nis} 
1 \to \ClS(\un{G}) \to H^1_\et(\CS,\un{G}) \xrightarrow{\lm} H^1(K,G) \times \prod_{\fp \notin S} H^1_\et(\hat{\CO}_\fp,\un{G}_\fp).      
\end{equation} 
The left exactness reflects the fact that $\ClS(\un{G})$ can be identified with the principal genus of $\un{G}$. 

If, furthermore, $\un{G}$ has the property: 
\begin{equation} \label{property}
\forall \fp \notin S : \ \ H^1_\et(\hat{\CO}_\fp,\un{G}_\fp) \hookrightarrow H^1_\et(\hat{K}_\fp,G_\fp),    
\end{equation}  
then sequence \eqref{Nis} is simplified to (cf. \cite[Cor.~3.6]{Nis}):  
\begin{equation} \label{Nis simple}
1 \to \ClS(\un{G}) \to H^1_\et(\CS,\un{G}) \xrightarrow{\lm_K} H^1(K,G),       
\end{equation}  
which indicates that 
any two $\un{G}$-torsors share the same genus if and only if they are $K$-isomorphic. 
If $\un{G}$ has connected fibers, then by Lang's Theorem $H^1_\et(\hat{\CO}_\fp,\un{G}_\fp)$ vanishes for any prime $\fp$ 
(see \cite[Ch.VI, Prop.5]{Ser} and recall that all residue fields are finite), thus $\un{G}$ has property \eqref{property}.  

\begin{remark} \label{Picard group is finite} 
The multiplicative $\CS$-group $\un{\BG}_m$ admits property \eqref{property} thus sequence \eqref{Nis simple},   
in which the rightmost term vanishes by Hilbert 90 Theorem.   
Hence the class set $\ClS(\un{\BG}_m)$, being finite as previously mentioned,  
is bijective as a pointed-set to $H^1_\et(\CS,\un{\BG}_m)$, 
which is identified with $\Pic(\CS)$ (cf. \cite[Cha.III,\S4]{Mil1}) thus being finite too. 
This holds true for any finite \'etale extension $R$ of $\CS$. 
\end{remark}

\begin{remark} \label{disconnected} 
If $\un{G}$ (locally of finite presentation) is disconnected but its connected component $\un{G}^0$ is reductive and $\un{G}/\un{G}^0$ is a finite representable group, 
then it admits again property \eqref{property} (see the proof of Proposition~3.14 in \cite{CGP}), 
thus sequence \eqref{Nis simple} as well. 
If, furthermore, for any $[\un{G}'] \in \ClS(\un{G})$, the map $G'(K) \to (G'/(G')^0)(K)$ is surjective, 
then $\ClS(\un{G}) = \ClS(\un{G}^0)$ (cf. \cite[Lemma~3.2]{Bit3}). 
\end{remark}

\begin{lem} \label{H1=1 sc} 
Let $\un{G}$ be a smooth and affine $\CS$-group scheme with connected fibers.   
Suppose that its generic fiber $G$ is almost simple, simply connected and $G_S$ is non-compact. 
Then $H^1_\et(\CS,\un{G})=1$. 
\end{lem}

\begin{proof}
The proof, basically relying on the strong approximation property related to $G$, 
is the one of Lemma 3.2 in \cite{Bit1}, replacing $\{\iy\}$ by $S$. 
\end{proof}

\subsection{The fundamental group: the quasi-split case}
The following is the Shapiro Lemma for the \'etale cohomology:

\begin{lem} \label{Shapiro}
Let $f:R \to S$ be a finite \'etale extension of schemes 
and $\G$ a smooth $R$-module. Then 
$\forall p: \ H^p_\et(S,\text{Res}_{R/S}(\G)) \cong H^p_\et(R,\G)$. 
\end{lem}

(See \cite[VIII, Cor.~5.6]{SGA4} in which the Leray spectral sequence for $R/S$ degenerates, 
whence the edge morphism $H^p_\et(S,\text{Res}_{R/S}(\G)) \to H^p_\et(R,\G)$ is an isomorphism.)  

\begin{remark} \label{finite etale extension is embedded in generic fiber} 
As $C$ is smooth, $\Sp \CS$ is normal, i.e., is integrally closed locally everywhere, 
thus any finite \'etale covering of $\CS$ arises by its normalization 
in some separable unramified extension of $K$ (e.g., \cite[Theorem~6.13]{Len}). 
\end{remark}

Assume $\un{F} = \text{Res}_{R/\CS}(\un{\mu}_m)$, $R$ is finite \'etale over $\CS$.  
Then the Shapiro Lemma (\ref{Shapiro}) with $p=2$ gives $H^2_\et(\CS,\un{F}) \cong H^2_\et(R,\un{\mu}_m)$. 
\'Etale cohomology applied to the Kummer sequence over $R$  
\begin{equation} \label{R Kummer} 
1 \to \un{\mu}_m \to \un{\BG}_m \xrightarrow{x \mapsto x^m} \un{\BG}_m \to 1 
\end{equation}   
gives rise to the exact sequences of abelian groups: 
\begin{align} \label{Kummer mu2 H1} 
1 &\to H^0_\et(R,\un{\mu}_m)  \to  R^\times                         \xrightarrow{\times m}   (R^\times)^m  \to 1,  \\ \nonumber
1 &\to R^\times/(R^\times)^m  \to  H^1_\et(R,\un{\mu}_m)            \to                    \Pic(R)[m] \to 1,  \\ \nonumber
1 &\to \Pic(R)/m              \to  H^2_\et(R,\un{\mu}_m)            \xrightarrow{i_*}      \Br(R)[m]  \to 1,  
\end{align}
in which as above $\Pic(R)$ is identified with $H^1_\et(R,\un{\BG}_m)$,   
and the Brauer group $\Br(R)$ -- classifying Azumaya $R$-algebras -- is identified with $H^2_\et(R,\un{\BG}_m)$ (cf. \cite[Cha.IV, \S 2]{Mil1}).  


\subsection{The fundamental group: the non quasi-split case} \label{subsection nqs}

The group $\un{F} = \text{Res}^{(1)}_{R/\CS}(\un{\mu}_m)$ fits into the short exact sequence of smooth $\CS$-groups 
(recall $\un{\mu}_m$ is assumed to be smooth as $m$ is prime to $\text{char}(K)$):  
\begin{equation*} 
1 \to \un{F} \to \text{Res}_{R/\CS}(\un{\mu}_m) \xrightarrow{N_{R/\CS}} \un{\mu}_m \to 1     
\end{equation*} 
which yields by \'etale cohomology together with Shapiro's isomorphism the long exact sequence:
\begin{equation} \label{LES nqs} 
 ... \to H^r_\et(\CS,\un{F}) \xrightarrow{I^{(r)}} H^r_\et(R,\un{\mu}_m) \xrightarrow{N^{(r)}} H^r_\et(\CS,\un{\mu}_m) \to H^{r+1}_\et(\CS,\un{F}) \to ... \ .
\end{equation}

\begin{notation} \label{[m] and /m}
For a group homomorphism $f:A \to B$, we denote by $f/m:A/m \to B/m$ and $f[m]:A[m] \to B[m]$ 
the canonical maps induced by $f$.  
\end{notation}

\begin{lem} \label{N surjective}
If $[R:\CS]$ is prime to $m$, then $N^{(r)},N^{(r)}[m]$ and $N^{(r)}/m$ are surjective for all $r \geq 0$.  
In particular, if $\un{F} = \text{Res}_{R/\CS}^{(1)}(\un{\mu}_m)$, then sequence \ref{LES nqs} induces an exact sequence of abelian groups 
for every $r \geq 0$: 
\begin{equation} \label{degree is prime to m}
1 \to H^r_\et(\CS,\un{F}) \xrightarrow{I^{(r)}} H^r_\et(R,\un{\mu}_m) \xrightarrow{N^{(r)}} H^r_\et(\CS,\un{\mu}_m) \to 1. 
\end{equation}
\end{lem}

\begin{proof}
The composition of the induced norm $N_{R/\CS}$ with the diagonal morphism coming from the Weil restriction 
\begin{equation} \label{composition} 
\un{\mu}_{m,\CS} \to \text{Res}_{R/\CS}(\un{\mu}_{m,R}) \xrightarrow{N_{R/\CS}} \un{\mu}_{m,\CS} 
\end{equation}
is the multiplication by $n := [R:\CS]$. 
It induces for every $r \geq 0$ the maps:
\begin{equation} \label{N} 
H^r_\et(\CS,\un{\mu}_m) \to H^r_\et(R,\un{\mu}_m) \xrightarrow{N^{(r)}} H^r_\et(\CS,\un{\mu}_m) 
\end{equation}
whose composition is again the multiplication by $n$ on $H^r_\et(\CS,\un{\mu}_m)$,  
being an automorphism when $n$ is prime to $m$.  
Hence $N^{(r)}$ is surjective for all $r \geq 0$. 

Replacing $\un{\mu}_m$ with $\un{\BG}_m$ in sequence \eqref{composition} and 
taking the $m$-torsion subgroups of the resulting cohomology sets, we get the group maps: 
$$ H^r_\et(\CS,\un{\BG}_m)[m] \to H^r_\et(R,\un{\BG}_m)[m] \xrightarrow{N^{(r)}[m]} H^r_\et(\CS,\un{\BG}_m)[m] $$
whose composition is multiplication by $n$ on $H^r_\et(\CS,\un{\BG}_m)[m]$, 
being an automorphism again as $n$ is prime to $m$, whence $N^{(r)}[m]$ is an epimorphism for every $r \geq 0$.  
The same argument applied to $N^{(r)}/m$ shows it is surjective for every $r \geq 0$ as well. 
\end{proof}

Back to the general case ($[R:\CS]$ does not have to be prime to $m$), 
applying the Snake lemma to the exact and commutative diagram of abelian groups: 
\begin{equation} \label{N^2 diagram} 
\xymatrix{                     
1 \ar[r] & \Pic(R)/m  \ar[r] \ar[d]^{N^{(1)}/m} & H^2_\et(R,\un{\mu}_m) \ar[r]^{i_*} \ar[d]^{N^{(2)}} & \Br(R)[m]   \ar[r] \ar[d]^{N^{(2)}[m]} & 1 \\ 
1 \ar[r] & \Pic(\CS)/m \ar[r]       & H^2_\et(\CS,\un{\mu}_m) \ar[r]                      & \Br(\CS)[m] \ar[r]        & 1
}
\end{equation}
yields an exact sequence of $m$-torsion abelian groups: 
\begin{align} \label{i_*'} 
1 &\to \ker(\Pic(R)/m \xrightarrow{N^{(1)}/m} \Pic(\CS)/m) \to \ker(N^{(2)}) \xrightarrow{i_*'} \ker(\Br(R)[m] \xrightarrow{N^{(2)}[m]} \Br(\CS)[m]) \\ \nonumber 
  &\to \cok(\Pic(R)/m \xrightarrow{N^{(1)}/m} \Pic(\CS)/m),  
\end{align}
where $i_*'$ is the restriction of $i_*$ to $\ker(N^{(2)})$. 
Together with the surjection $I^{(2)}:H^2_\et(\CS,\un{F}) \surj \ker(N^{(2)})$ coming from sequence \eqref{LES nqs}, we get the commutative diagram:
\begin{equation} \label{nqs diagram}
\xymatrix{                     
                                                                 & \ker \left(\Pic(R)/m \to \Pic(\CS)/m \right) \ar@{^{(}->}[d]      \\     
H^2_\et(\CS,\un{F}) \ar@{->>}[r]^-{I^{(2)}} \ar[rd]_-{i_*^{(1)}} & \ker \left(H^2_\et(R,\un{\mu}_m) \xrightarrow{N^{(2)}} H^2_\et(\CS,\un{\mu}_m) \right) \ar[d]^{i_*'}   \\ 
                                                                 & \ker \left(\Br(R)[m] \to \Br(\CS)[m] \right).   
} 
\end{equation}

\begin{prop} \label{i_*' surjective}
If $[R:\CS]$ is prime to $m$, then there exists a canonical exact sequence of abelian groups 
$$ 1 \to \ker \left(\Pic(R)/m \to \Pic(\CS)/m \right) \to \ker(N^{(2)}) \xrightarrow{i_*'} \ker \left(\Br(R)[m] \to \Br(\CS)[m] \right) \to 1. $$ 
\end{prop}

\begin{proof}
This sequence is the column in diagram \eqref{nqs diagram} 
since Lemma \ref{N surjective} shows the surjectivity of $N^{(1)}/m$, 
which in turn implies the surjectivity of $i_*'$ by the exactness of sequence \eqref{i_*'}. 
\end{proof}

Recall the definition of $i(\un{F})$ (Def. \ref{i}), and of the maps $i_*$ and $i_*^{(1)}$ (sequences \eqref{Kummer mu2 H1} and \eqref{nqs diagram}). 
  
\begin{definition} \label{i*}   
Let $\un{F}$ be one of the basic factors of an admissible fundamental group (see Def. \ref{admissible}).  
The map $\ov{i}_*:H^2_\et(\CS,\un{F}) \to i(\un{F})$ is defined as:
\begin{align*} 
\ov{i}_* := \left \{ \begin{array}{l l}
i_*        &  \un{F} = \text{Res}_{R/\CS}(\un{\mu}_m), \\
i_*^{(1)}  &  \un{F} = \text{Res}^{(1)}_{R/\CS}(\un{\mu}_m) \ \text{and} \ ([R:\CS],m)=1.   
\end{array}\right.   
\end{align*} 
More generally, if $\un{F} = \prod_{k=1}^r \un{F}_k$ where each $\un{F}_k$ is one of the above, we set it to be the composition:   
$$ \ov{i}_*:H^2_\et(\CS,\un{F}) \xrightarrow{\sim} \bigoplus_{k=1}^r H^2_\et(\CS,\un{F}_k) \xrightarrow{\bigoplus_{k=1}^r (\ov{i}_*)_k} i(\un{F}) = \prod_{k=1}^r i(\un{F}_k). $$ 
\end{definition}

\begin{cor} \label{admissible surjective}
If $\un{F}$ is admissible, then there exists a short exact sequence  
\begin{equation} \label{exact sequence} 
1 \to j(\un{F}) \to H^2_\et(\CS,\un{F}) \xrightarrow{\ov{i}_*} i(\un{F}) \to 1. 
\end{equation} 
\end{cor}

\begin{proof}
If $\un{F} = \text{Res}_{R/\CS}(\un{\mu}_m)$ then the sequence of the corollary is simply a restatement 
of the last sequence in \eqref{Kummer mu2 H1} by the definitions of $i(\un{F})$ and $j(\un{F})$ (see Definition \ref{i}). 
On the other hand, if $\un{F} = \text{Res}_{R/\CS}^{(1)}(\un{\mu}_m)$ with $[R:\CS]$ is prime to $m$, 
then $I^{(2)}$ induces an isomorphism of abelian groups $H^2_\et(\CS,\un{F}) \cong \ker(N^{(2)})$ 
by the exactness of \eqref{degree is prime to m} for $r=2$. 
Thus the sequence of the corollary is isomorphic to the sequence in Proposition \ref{i_*' surjective} again by the definitions of $j(\un{F})$ and $i(\un{F})$. 
The two cases considered above suffice to establish the corollary by the definition of admissible 
(see Def. \ref{admissible}) and the definition of $\ov{i}_*$ (see Def. \ref{i*}).  
\end{proof}

\begin{definition} \label{Euler}
Let $\un{X}$ be a constructible sheaf defined over $\Sp \CS$ and let $h_i(\un{X}) := |H^i_\et(\CS,\un{X})|$. 
The (restricted) \emph{Euler-Poincar\'e characteristic} of $\un{X}$ is defined to be (cf. \cite[Ch.II~\S 2]{Mil2}): 
$$ \chi_S(\un{X}) := \prod_{i=0}^2 h_i(\un{X})^{(-1)^i}. $$ 
\end{definition}

\begin{definition} \label{l}
Let $R$ be a finite \'etale extension of $\CS$. 
We define: 
\begin{align*} 
l(\un{F}) := \left \{ \begin{array}{l l}
\fc{|R^\times[m]|}{[ R^{\times}:(R^{\times})^m ]}     &  \un{F} = \text{Res}_{R/\CS}(\un{\mu}_m) \\ \\
\fc{|\ker(N^{(0)}[m])|}{|\ker(N^{(0)}/m)|}            &  \un{F} = \text{Res}^{(1)}_{R/\CS}(\un{\mu}_m).  
\end{array}\right.  
\end{align*} 
As usual, for $\un{F} = \prod_{k=1}^r \un{F}_k$ where each $\un{F}_k$ is one of the above, 
we put $l(\un{F})=\prod_{k=1}^r l(\un{F}_k)$. 
\end{definition}

\begin{lem} \label{abs almost simple non qs}
If $\un{F}$ is admissible then $\chi_S(\un{F}) = l(\un{F}) \cdot |i(\un{F})|$. 
\end{lem}
  
\begin{proof}
It is sufficient to check the assertion for the two basic types of (direct) factors: \\
Suppose $\un{F} = \text{Res}_{R/\CS}(\un{\mu}_m)$.  
Then sequences \eqref{Kummer mu2 H1} together with Shapiro's Lemma give  
\begin{align*}
h_i(\un{F}) = |H^i_\et(R,\un{\mu}_m)| =
\left \{ 
\begin{array}{l l}
|R^\times[m]|,                              & i=0 \\ 
\left[R^{\times}:(R^{\times})^m \right] \cdot |\Pic(R)[m]|, & i=1 \\
|\Pic(R)/m| \cdot |\Br(R)[m]|               & i=2.  
\end{array}\right. 
\end{align*}
So as $\Pic(R)$ is finite (see Remark \ref{Picard group is finite}), $|\Pic(R)[m]| = |\Pic(R)/m|$ and we get: 
\begin{align*} 
\chi_S(\un{F}) := \fc{h_0(\un{F}) \cdot h_2(\un{F})}{h_1(\un{F})} = \fc{|R^\times[m]| \cdot |\Pic(R)/m| \cdot |\Br(R)[m]|}{[ R^{\times}:(R^{\times})^m ] \cdot |\Pic(R)[m]|} 
              = l(\un{F}) \cdot |i(\un{F})|. 
\end{align*}

Now suppose $\un{F} = \text{Res}^{(1)}_{R/\CS}(\un{\mu}_m)$ such that $[R:\CS]$ is prime to $m$.  
By Lemma \ref{N surjective} $N^{(r)},N^{(r)}[m]$ and $N^{(r)}/m$ are surjective for all $r \geq 0$,  
so the long sequence \eqref{LES nqs} is cut into short exact sequences:
\begin{equation} \label{for all r}
\forall r \geq 0: \ 1 \to H^r_\et(\CS,\un{F}) \xrightarrow{I^{(r)}} H^r_\et(R,\un{\mu}_m) \xrightarrow{N^{(r)}} H^r_\et(\CS,\un{\mu}_m) \to 1   
\end{equation}
from which we see that (notice that $N^{(0)}[m]$ coincides with $N^{(0)}$):
\begin{equation} \label{H0(F)}
h_0(\un{F}) = |\ker(R^\times[m] \xrightarrow{N^{(0)}[m]} \CS^\times[m])|.   
\end{equation}
 
The Kummer exact sequences for $\un{\mu}_m$ defined over both $\CS$ and $R$ yield the exact diagram:
\begin{equation} \label{N2}
\xymatrix{
1 \ar[r] & R^\times   /(R^\times)^m   \ar[r] \ar@{->>}[d]^{N^{(0)}/m} & H^1_\et(R,\un{\mu}_m)  \ar@{->>}[d]^{N^{(1)}} \ar[r] & \Pic(R)[m]  \ar@{->>}[d]^{N^{(1)}[m]} \ar[r] & 1\\ 
1 \ar[r] & \CS^\times /(\CS^\times)^m \ar[r]                          & H^1_\et(\CS,\un{\mu}_m)                       \ar[r] & \Pic(\CS)[m]                          \ar[r] & 1
}
\end{equation}
from which we see together with sequence \eqref{for all r} that: 
$$ h_1(\un{F}) = |\ker(N^{(1)})| = |\ker(R^\times/(R^\times)^m \xrightarrow{N^{(0)}/m} \CS^\times /(\CS^\times)^m)| \cdot |\ker(\Pic(R)[m] \xrightarrow{N^{(1)}[m]} \Pic(\CS)[m])|. $$
Similarly, by sequence \eqref{for all r} and Proposition \ref{i_*' surjective} we find that: 
$$ h_2(\un{F}) = |\ker(N^{(2)})| = |\ker(\Pic(R)/m \xrightarrow{N^{(1)}/m} \Pic(\CS)/m)| \cdot |\ker(\Br(R)[m] \xrightarrow{N^{(2)}[m]} \Br(\CS)[m])|. $$
Altogether we get: 
\begin{align*} 
\chi_{S}(\un{F}) = \fc{h_0(\un{F}) \cdot h_2(\un{F})}{h_1(\un{F})} = \fc{|\ker(N^{(0)}[m])|}{|\ker(N^{(0)}/m)|} \cdot \fc{|\ker(N^{(1)}[m])|}{|\ker(N^{(1)}/m)|} \cdot |\ker(N^{(2)}[m])|.  
\end{align*}
The group of units $R^\times$ is a finitely generated abelian group (cf. \cite[Prop.~14.2]{Ros}), thus the quotient $R^\times / (R^\times)^m$ is a finite group.  
Since $\Pic(R)[m]$ is also finite, $\ker(N^{(1)})$ in diagram \eqref{N2} is finite, thus $|\ker(N^{(1)})[m]| = |\ker(N^{(1)})/m|$, and we are left with:
\begin{equation*} 
\chi_{S}(\un{F}) = \fc{|\ker(N^{(0)}[m])|}{|\ker(N^{(0)}/m)|} \cdot |\ker(N^{(2)}[m])|  = l(\un{F}) \cdot |i(\un{F})|.  \hfill \qedhere 
\end{equation*}
\end{proof}

\begin{remark}
The computation of $l(\un{F})$, for specific choices of $R$, $\CS$ and $m$, is an interesting (and probably open) problem. 
For example, when $\un{F}$ is not quasi-split, the denominator of this number is the order of the group of units of $R$  
whose norm down to $\CS$ is an $m$-th power of a unit in $\CS$, modulo $(R^\times)^m$. 
Such computations are hard to find in the literature, if they exist at all. 
\end{remark}

\bk

\section{The set of genera} \label{Section genera}
From now and on we assume $\un{G}$ is semisimple and that its fundamental group $\un{F}$ is of order prime to $\text{char}(K)$, thus smooth. 
\'Etale cohomology applied to the universal covering of $\un{G}$   
\begin{equation} \label{universal covering} 
1 \to \un{F} \to \un{G}^\sc \to \un{G} \to 1,     
\end{equation} 
gives rise to the exact sequence of pointed-sets:  
\begin{equation} \label{universal covering cohomology} 
H^1_\et(\CS,\un{G}^\sc) \to H^1_\et(\CS,\un{G}) \xrightarrow{\dl_{\un{G}}} H^2_\et(\CS,\un{F}) 
\end{equation}
in which the co-boundary map $\dl_{\un{G}}$ is surjective, as the domain $\CS$ is of Douai-type, 
implying that $H^2_\et(\CS,\un{G}^\sc)=1$ (see Definition~5.2 and Example~5.4 (iii) in \cite{Gon}). 

\begin{prop} \label{sequence of wG} 
There exists an exact sequence of pointed-sets:
$$ 1 \to \ClS(\un{G}) \xrightarrow{h} H^1_\et(\CS,\un{G}) \xrightarrow{w_{\un{G}}} i(\un{F}) $$
in which $h$ is injective. 
If $\un{F}$ is admissible, then $w_{\un{G}}$ is surjective. 
\end{prop}

\begin{proof} 
It is shown in \cite[Thm.~2.8 and proof of Thm.~3.5]{Nis} that there exist a canonical bijection 
$\a_{\un{G}} : H^1_{\text{Nis}}(\CS,\un{G}) \cong \ClS(\un{G})$ and a canonical injection 
$i_{\un{G}} : H^1_{\text{Nis}}(\CS,\un{G}) \hookrightarrow H^1_\et(\CS,\un{G})$ of pointed-sets (as Nisnevich's covers are \'etale). 
Then the map $h$ of the statement is the composition $i_{\un{G}} \circ \a_{\un{G}}^{-1}$. 

Assume $\un{F} = \text{Res}_{R/\CS}(\un{\mu}_m)$.  
The composition of the surjective map $\dl_{\un{G}}$ from \eqref{universal covering cohomology} 
with Shapiro's isomorphism and the surjective morphism $i_*$ from \eqref{Kummer mu2 H1},  
is a surjective $R$-map:  
\begin{equation} \label{witt invariant}
w_{\un{G}}: H^1_\et(\CS,\un{G}) \xrightarrowdbl{\dl_{\un{G}}} H^2_\et(\CS,\un{F}) \xrightarrow{\sim} H^2_\et(R,\un{\mu}_m) \xrightarrowdbl{i_*} \Br(R)[m].  
\end{equation}
On the generic fiber, since $G^\sc:=\un{G}^\sc \otimes_{\CS} K$ is simply connected, 
$H^1(K,G^\sc)$ vanishes due to Harder (cf. \cite[Satz~A]{Har}), as well as its other $K$-forms (this would not be true, however, if $K$ were a number field with real places).  
So Galois cohomology applied to the universal $K$-covering   
\begin{equation} \label{K universal covering} 
1 \to F \to G^\sc \to G \to 1 
\end{equation}
yields an embedding of pointed-sets $\dl_{G}:H^1(K,G) \hookrightarrow H^2(K,F)$, which is also surjective as $K$ is of Douai-type as well.  
The extension $R$ of $\CS$ arises from an unramified Galois extension $L$ of $K$ by Remark \ref{finite etale extension is embedded in generic fiber},  
and Galois cohomology applied to the Kummer exact sequence of $L$-groups
$$ 1 \to \mu_m \to \BG_m \xrightarrow{x \mapsto x^m} \BG_m \to 1 $$
yields, together with Shapiro's Lemma $H^2(K,F)\cong H^2(L,\un{\mu}_m)$ and Hilbert 90 Theorem, 
the identification $(i_*)_{L}: H^2(K,F) \cong \Br(L)[m]$,   
whence the composition $(i_*)_{L} \circ \dl_{G}$ is an injective $L$-map:  
$$ w_{G}:  H^1(K,G) \xhookrightarrow{\dl_G} H^2(K,F) \stackrel{(i_*)_{L}}{\cong} \Br(L)[m]. $$
Now we know due to Grothendieck that $\Br(R)$ is a subgroup of $\Br(L)$ 
(see \cite[Prop.~2.1]{Gro} and \cite[Example~2.22, case (a)]{Mil1}). 
Altogether we retrieve the commutative diagram of pointed-sets: 
\begin{equation} \label{Witt diagram qs}
\xymatrix{
H^1_\et(\CS,\un{G})   \ar@{->>}[r]^{w_{\un{G}}} \ar[d]^{\lm_K}  & \Br(R)[m] \ar@{^{(}->}[d]^{j}  \\
H^1(K,G)              \ar@{^{(}->}[r]^{w_G}                     & \Br(L)[m],  
}
\end{equation}
from which, together with sequence \eqref{Nis simple} (recall $\un{G}$ has connected fibers), we may observe that: 
$$ \ClS(\un{G}) = \ker(\lm_K) = \ker(w_{\un{G}}). $$ 

When $\un{F} = \text{Res}^{(1)}_{R/\CS}(\un{\mu}_m)$,  
we define the map $w_{\un{G}}$ using diagram \ref{nqs diagram} to be the composition 
\begin{equation} \label{w_G nqs} 
w_{\un{G}} : H^1_\et(\CS,\un{G}) \xrightarrow{\dl_{\un{G}}} H^2_\et(\CS,\un{F}) \xrightarrow{i_*^{(1)}} \ker \left(\Br(R)[m] \xrightarrow{N^{(2)}[m]} \Br(\CS)[m] \right)  
\end{equation}
being surjective by Corollary \ref{admissible surjective} given that $[R:\CS]$ is prime to $m$. 
On the generic fiber, Galois cohomology with Hilbert 90 Theorem give:   
$$ w_{G}: H^1(K,G) \xhookrightarrow{\dl_G} H^2(K,F) \stackrel{(i^{(1)}_*)_K}{\cong} \ker \left(\Br(L)[m] \xrightarrow{N^{(2)}_L[m]} \Br(K)[m] \right). $$
This time we get the commutative diagram of pointed sets: 
\begin{equation} \label{Witt diagram nqs}
\xymatrix{
H^1_\et(\CS,\un{G})   \ar[r]^-{w_{\un{G}}} \ar[d]^{\lm_K}  & \ker \left(\Br(R)[m]  \xrightarrow{N^{(2)}[m]}     \Br(\CS)[m] \right) \ar@{^{(}->}[d]^{j}  \\
H^1(K,G)              \ar@{^{(}->}[r]^-{w_G}               & \ker \left(\Br(L)[m]  \xrightarrow{(N^{(2)}[m])_L} \Br(K)[m]   \right),  
}
\end{equation}
from which we may deduce again that: 
$$ \ClS(\un{G}) \stackrel{\eqref{Nis simple}}{=} \ker(\lm_K) = \ker(w_{\un{G}}). $$

More generally, if $\un{F}$ is a direct product of such basic factors,  
then as the cohomology sets commute with direct products, 
the target groups of $w_{\un{G}}$ and $w_G$ become the product of the target groups of their factors, 
and the same argument gives the last assertion. 
\end{proof}

\begin{cor} \label{genera} 
There is an injection of pointed sets $w_{\un{G}}': \text{gen}(\un{G}) \hookrightarrow i(\un{F})$. \\ 
If $\un{F}$ is admissible then $w_{\un{G}}'$ is a bijection. 
In particular if $\un{F}$ is split, then $|\text{gen}(\un{G})| = |F|^{|S|-1}$.    
\end{cor}

\begin{proof}
The commutativity of diagrams \eqref{Witt diagram qs} and \eqref{Witt diagram nqs}
and the injectivity of the map $j$ in them show that $w_{\un{G}}$ is constant on each fiber of $\lm_K$, i.e., on the genera of $\un{G}$.  
Thus $w_G$ induces a map (see Proposition \ref{sequence of wG}): 
$$ w_{\un{G}}' : \text{gen}(\un{G}) \to \Im(w_{\un{G}}) \subseteq i(\un{F}). $$
These diagrams commutativity together with the injectivity of $w_G$ imply the injectivity of $w_{\un{G}}'$. 

If $\un{F}$ is admissible then $\Im(w_{\un{G}}) = i(\un{F})$. 
In particular if $\un{F}$ is split, then $\gen(\un{G}) \cong \prod_{i=1}^r \Br(\CS)[m_i]$.  
It is shown in the proof of \cite[Lemma~2.2]{Bit1} that $\Br(\CS) = \ker \left(\Q/\Z \xrightarrow{\sum_{\fp \in S}\text{Cor}_\fp} \Q/\Z \right)$ 
where $\text{Cor}_\fp$ is the corestriction map at $\fp$.   
So $|\Br(\CS)[m_i]| = m_i^{|S|-1}$ for all $i$ and the last assertion~follows. 
\end{proof}

The following table refers to $\CS$-group schemes whose generic fibers are split, absolutely almost simple and adjoint. 
The right column is Corollary ~\ref{genera}:  

\bk

\begin{center}
{\small
 \begin{tabular}{|c | c | c | } 
 \hline
 Type of $\un{G}$           & $\un{F}$                     & \# $\text{gen}(\un{G})$  \\ \hline  \hline                                   
  ${^1}\text{A}_{n-1}$        & $\un{\mu}_n$                 & $n^{|S|-1}$   \\ \hline                                          
  $\text{B}_n,\text{C}_n,\text{E}_7$        & $\un{\mu}_2$                 & $2^{|S|-1}$   \\ \hline                                                                                             
	${^1}\text{D}_n$   & \begin{tabular}[x]{@{}c@{}}$\un{\mu}_4, \ n=2k+1$ \\ $\un{\mu}_2 \times \un{\mu}_2, \ n=2k$ \end{tabular} & $4^{|S|-1}$ \\ \hline 
	${^1}\text{E}_6$            & $\un{\mu}_3$                 & $3^{|S|-1}$ \\ \hline
  $\text{E}_8,\text{F}_4,\text{G}_2$   & $1$ & $1$ \\ \hline
\end{tabular} }
\end{center}

\begin{lem} \label{H1G iso H2F}
Let $\un{G}$ be a semisimple and almost simple $\CS$-group not of (absolute) type $\text{A}$,    
then $H^1_\et(\CS,\un{G})$ bijects as a pointed-set to the abelian group $H^2_\et(\CS,\un{F})$. 
\end{lem}

\begin{proof} 
Since $G^\sc$ is not of (absolute) type $\text{A}$, it is locally isotropic everywhere (\cite[4.3~and~4.4]{BT}), 
whence $\ker(\dl_{\un{G}}) \subseteq H^1_\et(\CS,\un{G}^\sc)$ vanishes due to Lemma \ref{H1=1 sc}. 
Moreover, for any $\un{G}$-torsor $P$, the base-point change: $\un{G} \mapsto {^P}\un{G}$ defines a bijection of pointed-sets: 
$H^1_\et(\CS,\un{G}) \to H^1_\et(\CS,{^P}\un{G})$ (see in Section \ref{Introduction}).   
But ${^P}\un{G}$ is an inner form of $\un{G}$, thus not of type $\text{A}$ as well, 
hence also $H^1_\et(\CS,({^P}\un{G})^\sc)=1$.     
We get that all fibers of $\dl_{\un{G}}$ in \eqref{universal covering cohomology} are trivial, 
which together with the surjectivity of $\dl_{\un{G}}$ amounts to the asserted bijection.  
\end{proof}

In other words, the fact that $\un{G}$ is not of (absolute) type $\text{A}$ 
guarantees that not only $G^\sc$, but also the universal covering of the generic fiber of inner forms of $\un{G}$      
of other genera are locally isotropic everywhere. 
This provides $H^1_\et(\CS,\un{G})$ the structure of an abelian group. 

\begin{cor} \label{the same cardinality}
If $\un{G}$ is not of (absolute) type $\text{A}$, 
then all its genera share the same cardinality. 
\end{cor}

\begin{proof} 
The map $w_{\un{G}}$ factors through $\dl_{\un{G}}$ (see \eqref{witt invariant} and \eqref{w_G nqs}) 
which is a bijection of pointed-sets in this case by Lemma \ref{H1G iso H2F}.  
So writing: $w_{\un{G}} = \ov{w}_{\un{G}} \circ \dl_{\un{G}}$.  
we get due to Proposition \ref{sequence of wG} the exact sequence of pointed-sets (a-priory, abelian groups): 
\begin{equation*} 
1 \to \ClS(\un{G}) \to H^2_\et(\CS,\un{F}) \xrightarrow{\ov{w}_{\un{G}}} i(\un{F}) 
\end{equation*}
in which all genera, corresponding to the fibers of $w_{\un{G}}$, are of the same cardinality.  
\end{proof}

Following E. Artin in \cite{Art}, we shall say that a Galois extension $L$ of $K$ is \emph{imaginary}  
if no prime of $K$ is decomposed into distinct primes in $L$. 

\begin{remark} \label{imaginary}
If $\un{G}$ is of (absolute) type $\text{A}$, but $S=\{\iy\}$, $G$ is $\hat{K}_\iy$-isotropic,  
and $F$ splits over an imaginary extension of $K$,    
then $H^1_\et(\CS,\un{G})$ still bijects as a pointed-set to $H^2_\et(\CS,\un{F})$. 
\end{remark}

\begin{proof}
As aforementioned, removing one closed point of a projective curve,   
the resulting Hasse domain has a trivial Brauer group. 
Thus $\Br(\CS=\Oiy)=1$, and as $F$ splits over an imaginary extension $L = \BF_q(C')$, 
corresponding to an \'etale extension $R = \BF_q[C'- \{ \iy' \}]$ of $\Oiy$ (see Remark \ref{finite etale extension is embedded in generic fiber})   
where $\iy'$ is the unique prime of $L$ lying above $\iy$, $\Br(R)$ remains trivial. 
This implies by Corollary \ref{genera} that $\un{G}$ has only one genus, namely, the principal one,  
in which the generic fibers of all representatives (being $K$-isomorphic to $G$) are isotropic at $\iy$.  
Then the resulting vanishing of $\ker(\dl_{\un{G}}) \subseteq H^1_\et(\CS,\un{G}^\sc)$ due to Lemma \ref{H1=1 sc}   
is equivalent to the injectivity of $\dl_{\un{G}}$. 
\end{proof}

The following general framework due to Giraud (see \cite[\S 2.2.4]{CF}), 
gives an interpretation of the $\un{G}$-torsors which may help us describe $w_{\un{G}}$  
more concretely. 

\begin{prop} \label{flat classification}
Let $R$ be a scheme and $X_0$ be an $R$-form, namely, 
an object of a fibered category of schemes defined over $R$. 
Let $\textbf{Aut}_{X_0}$ be its $R$-group of automorphisms. 
Let $\mathfrak{Forms}(X_0)$ be the category of $R$-forms that are locally isomorphic for some topology to $X_0$   
and let $\mathfrak{Tors}(\text{Aut}_{X_0})$ be the category of $\text{Aut}_{X_0}$-torsors in that topology.    
The functor 
$$ \p:\mathfrak{Forms}(X_0) \to \mathfrak{Tors}(\textbf{Aut}_{X_0}): \ X \mapsto \textbf{Iso}_{X_0,X} $$ 
is an equivalence of fibered categories.     
\end{prop}

\begin{example} 
Let $(V,q)$ be a regular quadratic $\CS$-space of rank $n \geq 3$    
and let $\un{G}$ be the associated \emph{special orthogonal group} $\SOV$ (see \cite[Definition~1.6]{Con1}). 
It is smooth and connected (cf. \cite[Theorem~1.7]{Con1}), 
and its generic fiber is of type $\text{B}_n$ if $\text{rank}(V)$ is odd, 
and of type ${^1}\text{D}_n$ otherwise. 
In both cases $\un{F} = \un{\mu}_2$, so we assume $\text{char}(K)$ is odd. 
Any such quadratic regular $\CS$-space $(V',q')$ of rank $n$ gives rise to a $\un{G}$-torsor $P$ by 
$$ V' \mapsto P = \textbf{Iso}_{V,V'} $$
where an isomorphism $A:V \to V'$ is a \emph{proper} $q$-isometry, i.e., such that $q' \circ A = q$ and $\det(A)=~1$. 
So $H^1_\et(\CS,\un{G})$ properly classifies regular quadratic $\CS$-spaces that are locally isomorphic to $(V,q)$ in the \'etale topology.  
Then $\dl_{\un{G}}([P])$ is the \emph{second Stiefel-Whitney} class of $P$ in $H^2_\et(\CS,\un{\mu}_2)$, 
classifying $\CS$-Azumaya algebras with involutions (see Def.~1, Remark~3.3 and Prop.~4.5 in \cite{Bit2}),      
and  
$$ w_{\un{G}}([\un{\textbf{SO}}_{q'}]) = 
\left \{ \begin{array}{l l}
[\textbf{C}_0(q')] - [\textbf{C}_0(q)] \in \Br(\CS)[2]  \ &  \ n \ \text{is odd} \\ \notag  
[\textbf{C}(q')] - [\textbf{C}(q)] \in \Br(\CS)[2] &  \ n \ \text{is even} 
\end{array}\right., $$ 
where $\textbf{C}(q)$ and $\textbf{C}_0(q)$ are the Clifford algebra of $q$ and its even part, respectively.     
\end{example}

\begin{example}
Let $\un{G}=\un{\textbf{PGL}}_n$ for $n \geq 2$.   
It is smooth and connected (\cite[Lemma~3.3.1]{Con2}) with $\un{F} = \un{\mu}_n$, so we assume $(\text{char}(K),n)=1$. 
For any projective $\CS$-space of rank $n$, 
by the Skolem-Noether Theorem for unital rings (see \cite[p.145]{Knus}) $\un{\textbf{PGL}}(V) = \textbf{Aut}(\text{End}_{\CS}(V))$. 
It is an inner form of $\un{G}$ obtained for $V = \CS^n$.   
So the pointed-set $H^1_\et(\CS,\un{G})$ classifies the projective $\CS$-modules of rank $n$ up to invertible $\CS$-modules. 
Given such a projective $\CS$-module $V$, the Azumaya $\CS$-algebra $A = \End_{\CS}(V)$ of rank $n^2$   
corresponds to a $\un{G}$-torsor by (see \cite[V,Remarque~4.2]{Gir}):
$$ A \mapsto P = \textbf{Iso}_{\un{M}_n,A} $$
where $\un{M}_n$ is the $\CS$-sheaf of $n \times n$ matrices. Here $w_{\un{G}}([P])=[A]$ in $\Br(\CS)[n]$. 
\end{example}

\bk

\section{The principal genus} \label{Section genus}
In this section, we study the structure of the principal genus $\ClS(\un{G})$. 

\begin{theorem} \label{genus isotropic}
If $\un{F}$ is admissible then there exists a surjection of pointed-sets  
$$ \psi_{\un{G}} : \ClS(\un{G}) \twoheadrightarrow j(\un{F}), $$
being a bijection provided that $G_S$ is non-compact (e.g., $G$ is not anisotropic of type $\text{A}$). 
\end{theorem}

\begin{proof} 
Combining the two epimorphisms 
-- $w_{\un{G}}$ defined in Prop.~\ref{sequence of wG} and $\dl_{\un{G}}$ described in Section~\ref{Section genera} --   
together with the exact sequence \eqref{exact sequence}, yields the exact and commutative diagram: 
\begin{equation} \label{genus diagram}
\xymatrix{
          & 1 \ar[r]  \ar[d]             & H^1_\et(\CS,\un{G}) \ar@{->>}[d]^{\dl_{\un{G}}} \ar@{=}[r]  & H^1_\et(\CS,\un{G}) \ar@{->>}[d]^{w_{\un{G}}} \ar[r] & 1   \\   
1  \ar[r] & j(\un{F}) \ar[r]^-{\partial} & H^2_\et(\CS,\un{F}) \ar[r]^-{\ov{i}_*}                      & i(\un{F})           \ar[r]                           & 1   \\   
}
\end{equation}
in which $\ker(w_{\un{G}}) = \ClS(\un{G})$.  
We imitate the Snake Lemma argument (the diagram terms are not necessarily all groups): 
for any $[H] \in \ClS(\un{G})$ one has $\ov{i}_*(\dl_{\un{G}}([H]))=[0]$, 
i.e., $\dl_{\un{G}}([H])$ has a $\partial$-preimage in $j(\un{F})$ 
which is unique as $\partial$ is a monomorphism of groups. 
This constructed map denoted $\psi_{\un{G}}$ gives rise to an exact sequence of pointed-sets:  
$$ 1 \to \mathfrak{K} 
\to \ClS(\un{G}) \xrightarrow{\psi_{\un{G}}} j(\un{F}) \to 1. $$
If $G_S$ is non-compact, then for any $[\un{H}] \in \ClS(\un{G})$ the generic fiber $H$ is $K$-isomorphic to $G$  
thus $H_S$ is non-compact as well, thus $\ker(H^1_\et(\CS,\un{H}) \xrightarrow{\dl_{\un{H}}} H^2_\et(\CS,\un{F})) \subseteq H^1_\et(\CS,\un{H}^\sc)$   
vanishes by Lemma \ref{H1=1 sc}.     
This means that $\dl_{\un{G}}$ restricted to $\ClS(\un{G})$ is an embedding, 
so $\mathfrak{K}=1$ and $\psi_{\un{G}}$ is a bijection.    
\end{proof}

\begin{remark}
The description of $\ClS(\un{G})$ in Theorem \ref{genus isotropic} holds true also for a disconnected group $\un{G}$ (where $\un{F}$ is the fundamental group of $\un{G}^0$), 
under the hypotheses of Remark \ref{disconnected}. 
\end{remark} 

\begin{definition} 
We say that the \emph{local-global Hasse principle} holds for $\un{G}$ if $h_S(\un{G})=1$. 
\end{definition}

This property means (when $\un{G}$ is connected) that a $\un{G}$-torsor is $\CS$-isomorphic to $\un{G}$ if and only if its generic fiber is $K$-isomorphic to $G$. 
Recall the definition of $j(\un{F})$ from Def.~\eqref{i}.

\begin{cor} \label{criterion}
Suppose $\un{F} \cong \prod_{i=1}^r \text{Res}_{R_i/\CS}(\un{\mu}_{m_i})$ where $R_i$ are finite \'etale extensions of $\CS$.  
If $G_S$ is non-compact, then the Hasse principle holds for $\un{G}$ if and only if $\forall i: (|\Pic(R_i)|,m_i)=1$. 
Otherwise ($G_S$ is compact), 
this principle holds for $\un{G}$ only if $\forall i:(|\Pic(R_i)|,m_i)=1$. \\ 
More generally, if $\un{F}$ is admissible and $G_S$ is non-compact, 
then this principle holds for $\un{G}$ provided that 
for each factor of the form $\text{Res}_{R/\CS}(\un{\mu}_m)$ or $\text{Res}^{(1)}_{R/\CS}(\un{\mu}_m)$ one has: $(|\Pic(R)|,m)=1$.  
\end{cor}

\begin{example} 
If $C^\af$ is an affine non-singular $\BF_q$-curve of the form $y^2=x^3+ax+b$,  
i.e., obtained by removing some $\BF_q$-rational point $\iy$ from an elliptic (projective) $\BF_q$-curve $C$,   
then $\Pic(C^\af) = \Pic(\Oiy) \cong C(\BF_q)$ (cf. e.g., \cite[Example~4.8]{Bit1}). 
Let again $\un{G} = \un{\textbf{PGL}}_n$ such that $(\text{char}(K),n)=~1$. 
As $|S|=1$ and $\un{F}$ is split, 
$\un{G}$ admits a single genus (Corollary \ref{genera}), 
which means that all projective $\Oiy$-modules of rank $n$ are $K$-isomorphic.  
If $\un{G}$ is $K$-isotropic, according to Theorem \ref{genus isotropic},  
there are exactly $|C^\af(\BF_q)/2|$ $\Oiy$-isomorphism classes of such modules,   
so the Hasse principle fails for $\un{G}$ if and only if $|C^\af(\BF_q)|$ is even.  
This occurs exactly when $C^\af$ has at least one $\BF_q$-point on the $x$-axis (thus of order $2$).  

On the other hand, take $\CS=\BF_3[t,t^{-1}]$ obtained by removing $S=\{t,t^{-1}\}$ from the projective $\BF_3$-line,   
and $\un{G}=\un{\textbf{PGL}}_n$ to be rationally isotropic over $\CS$: for example for $n=2$,  
it is isomorphic to the special orthogonal group of the standard split $\CS$-form $q_3(x_1,x_2,x_3)=x_1x_2+x_3^2$.  
Then as $q_3$ is rationally isotropic over $\CS$ (e.g., $q_3(1,2,1)=0$) and $\CS$ is a UFD,  
according to Corollary~\ref{criterion} 
the Hasse-principle holds for $\un{G}$ and there are two genera as $|F|=|S|=2$ (Cor. \ref{genera}).  
\end{example}

\begin{example} \label{non-split D}
Let $(V,q)$ be an $\CS$-regular quadratic form of even rank $n=2k \geq 4$ 
and let $\un{G} = \text{Res}_{R/\CS}(\SOV)$ where $R$ is finite \'etale over $\CS$. 
Then $\un{F} = \text{Res}_{R/\CS}(\un{\mu}_2)$,  
whence according to Corollary \ref{genera}, $\text{gen}(\un{G}) \cong \Br(R)[2]$. 
As $G$ and its twisted $K$-forms are $K$-isotropic (e.g., \cite[p.352]{PR}), 
each genus of $q$ contains exactly $\Pic(R)/2$ elements.  
\end{example}

\begin{example} \label{non-split A} 
Let $C'$ be an elliptic $\BF_q$-curve and $(C')^\af := C' - \{\iy'\}$. 
Then $R := \BF_q[(C')^\af]$ is a quadratic extension of $\Oiy = \BF_q[x]$ where $\iy = (1/x)$ and $\iy'$ is the unique prime lying above $\iy$, 
thus $L:= R \otimes_{\Oiy} K$ is imaginary over $K$. 
Let $\un{G} = \text{Res}_{R/\Oiy}(\un{\textbf{PGL}}_m)$, $m$ is odd and prime to $q$.  
Then $\un{F} = \text{Res}^{(1)}_{R/\Oiy}(\un{\mu}_m)$ is smooth, 
and $\un{G}$ is smooth and quasi-split as well as its generic fiber, thus is $K$-isotropic. 
By Remark \ref{imaginary} and sequence \eqref{exact sequence}, we get (notice that $\Oiy$ is a PID and that $\Br(R)=1$): 
$$ \ClS(\un{G}) = H^1_\et(\Oiy,\un{G}) \cong H^2_\et(\Oiy,\un{F}) \cong \ker(\Pic(R)/m \to \Pic(\Oiy)/m) = \Pic(R)/m. $$      
Hence the Hasse-principle holds for $\un{G}$ if and only if $|\Pic(R)|=|C'(\BF_q)|$ is prime to $m$. 
\end{example}

\bk

\section{The Tamagawa number of twisted groups} \label{Section application}
In this section we start with the generic fiber.   
Let $G$ be a semisimple group defined over a global field $K = \BF_q(C)$ with fundamental group $F$.  
The \emph{Tamagawa number} $\tau(G)$ of $G$ is defined as the covolume of the group $G(K)$ 
in the adelic group $G(\mathbb{A})$ (embedded diagonally as a discrete subgroup),  
with respect to the Tamagawa measure (see \cite{Weil}). 
T. Ono has established in \cite{Ono} a formula for the computation of $\tau(G)$ in case $K$ is an algebraic number field,   
which was later proved by Behrend and Dhillon in \cite[Theorem~6.1]{BD} also in the function field case: 
\begin{equation} \label{Ono}
\tau(G) = \fc{|\wh{F}^\fg|}{|\Sh^1(\wh{F})|} 
\end{equation}
where $\wh{F} := \Hom(F \otimes K^s,\BG_m)$, $\fg$ is the absolute Galois group of $K$, 
and $\Sh^1(\wh{F})$ is the first Shafarevitch--Tate group assigned to $\wh{F}$ over $K$. 
As a result, if $F$ is split, then  $\tau(G)=|F|$. 
So our main innovation, based on the above results and the following ones, 
would be simplifying the computation of $\tau(G)$ in case $F$ is not split, as may occur when $G$ is a twisted group. 

\bk

The following construction, as described in \cite{BK} and briefly revised here, 
expresses the global invariant $\tau(G)$ using some local data. 
Suppose $G$ is almost simple defined over the above $K=\BF_q(C)$, not anisotropic of type $\text{A}$, 
such that $(|F|,\text{char}(K))=1$. 
We remove one arbitrary closed point $\iy$ from $C$ and refer as above to the integral domain $\CS = \Oiy$. 
At any prime $\fp \neq \iy$, we consider the Bruhat-Tits $\CO_\fp$-model of $G_\fp$ corresponding to some special vertex in its associated building. 
Patching all these $\CO_\fp$-models along the generic fiber  
results in an affine and smooth $\Oiy$-model $\un{G}$ of $G$ (see \cite[\S 5]{BK}). 
It may be locally disconnected only at places that ramify over a minimal splitting field $L$ of $G$ (cf. \cite[4.6.22]{BT}).  

\bk

Denote $\A_\iy := \A_{\{\iy\}} =  \hat{K}_\iy \times \prod_{\fp \neq \iy} \hat{\CO}_\fp \subset \A$. 
Then $\un{G}(\A_\iy) G(K)$ is a normal subgroup of $\un{G}(\A)$ (cf. \cite[Thm.~3.2~3]{Tha}).  
The set of places $\text{Ram}_G$ that ramify in $L$ is finite, 
thus by the Borel density theorem (e.g., \cite[Thm.~2.4,~Prop.~2.8]{CM}),  
$\un{G}(\CO_{\{\iy\} \cup \text{Ram}_G})$ is Zariski-dense in $\prod_{\fp \in \text{Ram}_G \backslash \{ \iy\}} \un{G}_\fp$. 
This implies that $\un{G}(\A_\iy)G(K) = \un{G}^0(\A_\iy)G(K)$,   
where $\un{G}^0$ is the connected component of $\un{G}$. 

Since all fibers of $\varphi$ are isomorphic to $\ker(\varphi)$, we get a bijection of measure spaces   
\begin{align} \label{decompositionfirst}
\un{G}(\A)/G(K) &\cong \left( \un{G}(\A) / \un{G}(\A_\iy)G(K) \right) \times \left( \un{G}(\A_\iy)G(K) / G(K) \right) \\ \notag
                &= \left( \un{G}^0(\A) / \un{G}^0(\A_\iy)G(K) \right) \times \left( \un{G}^0(\A_\iy)G(K) / G(K) \right) \\ \notag
                &\cong \text{Cl}_{\{\iy\}}(\un{G}^0) \times \left( \un{G}^0(\A_\iy) / \un{G}^0(\A_\iy) \cap G(K) \right)  
\end{align}  
in which the left factor cardinality is the finite index $h_\iy(G):= h_{\{\iy\}}(\un{G}^0)$ (see Section \ref{Section: class set}),    
and in the right factor $\un{G}^0(\A_\iy) \cap G(K) = \un{G}^0(\Oiy)$. 
Due to the Weil conjecture stating that $\tau(G^\sc)=1$,  
as was recently proved in the function field case by Gaistgory and Lurie (see \cite[(2.4)]{Lur}),  
applying the Tamagawa measure $\tau$ on these spaces results in the Main Theorem in \cite{BK}:

\begin{theorem} \label{tau G} 
Let $\fg_\iy = \text{Gal}(\hat{K}_\iy^s/\hat{K}_\iy)$ be the Galois absolute group, 
$F_\iy:=\ker(G^\sc_\iy \to G_\iy)$,  
$\un{F}:= \ker(\un{G}^\sc \to \un{G})$ whose order is prime to $\text{char}(K)$,  
and $\wh{F_\iy} := \Hom(F_\iy \otimes \hat{K}_\iy^s,\BG_{m,\hat{K}_\iy^s})$. 
Then 
$$ \tau(G) =  h_\iy(G) \cdot \fc{t_\iy(G)}{j_\iy(G)}, $$  
where  
$t_\iy(G) = |\wh{F_\iy}^{\fg_\iy}|$ is the number of types in one orbit of a special vertex, 
in the Bruhat--Tits building associated to $G_\iy(\hat{K}_\iy)$,   
and $j_\iy(G) = h_1(\un{F}) / h_0(\un{F})$.   
\end{theorem}

We adopt Definition \ref{admissible} of being admissible to $F$, with a Galois extension $L/K$ replacing $R/\CS$. 
If $\un{G}$ is not of (absolute) type $\text{A}$ and $F$ is admissible, 
then due to the above results Theorem \ref{tau G} can be reformulated involving the fundamental group data only: 

\begin{theorem} \label{tau G 2} 
Let $G$ be an almost-simple group not of (absolute) type $\text{A}$ defined over $K=\BF_q(C)$ 
with an admissible fundamental group $F$ whose order is prime to $\text{char}(K)$.  
Then for any choice of a prime $\iy$ of $K$ one has: 
$$ \tau(G) = \fc{\chi_{\{\iy\}}(\un{F})}{|i(\un{F})|} \cdot |\wh{F_\iy}^{\fg_\iy}| = l(\un{F}) \cdot |\wh{F_\iy}^{\fg_\iy}|, $$  
where $\chi_{\{\iy\}}(\un{F})$ is the (restricted) Euler-Poincar\'e characteristic (cf. Definition \ref{Euler}),  
$i(\un{F})$ and $l(\un{F})$ are as in Definitions \ref{i} and \ref{l}, respectively, and the right factor is a local invariant.   
\end{theorem}  

\begin{proof}
If $G$ is not of (absolute) type $\text{A}$, according to Corollary \ref{the same cardinality} all genera of $\un{G}$ have the same cardinality. 
By Lemma \ref{H1G iso H2F} and Corollary \ref{genera} ($\un{F}$ is admissible as $F$ is, see Remark \ref{finite etale extension is embedded in generic fiber}) we then get  
$$ h_\iy(G) = |\text{Cl}_{\{\iy\}}(\un{G})| = \fc{|H^1_\et(\Oiy,\un{G})|}{|\text{gen}(\un{G})|} = \fc{h_2(\un{F})}{|i(\un{F})|}. $$
Now the first asserted equality follows from Theorem \ref{tau G} together with Definition \ref{Euler}: 
\begin{equation*}
\tau(G) = 1/j_\iy(\un{G}) \cdot h_\iy(\un{G}) \cdot t_\iy(G) 
        = \fc{h_0(\un{F})}{h_1(\un{F})} \cdot \fc{h_2(\un{F})}{|i(\un{F})|} \cdot |\wh{F_\iy}^{\fg_\iy}|
        = \fc{\chi_{\{\iy\}}(\un{F})}{|i(\un{F})|} \cdot |\wh{F_\iy}^{\fg_\iy}|. 
\end{equation*}
The rest is Lemma \ref{abs almost simple non qs}.  
\end{proof}

\begin{remark} \label{density}
By the geometric version of \v{C}ebotarev's density theorem (see in \cite{Jar}), 
there exists a closed point $\iy$ on $C$ at which $G_\iy$ is split. 
We shall call such a point a \emph{splitting point} of $G$. 
\end{remark} 

\begin{cor} \label{quasi split group}
Let $G$ be an adjoint group defined over $K=\BF_q(C)$ with fundamental group $F$ whose order is prime to $\text{char}(K)$ 
and whose splitting field is $L$. 
Choose some splitting point $\iy$ of $G$ on $C$ and let $R$ be a minimal \'etale~extension of $\Oiy := \BF_q[C-\{\iy\}]$ such that $R \otimes_{\Oiy} K = L$.   
Let $N^{(0)}:R^\times \to \Oiy^\times$ be the induced norm. Then: 
\begin{itemize}
\item[(1)] If $G$ is of type ${^2}\text{D}_{2k}$ then $\tau(G) = \fc{|R^\times[2]|}{[R^\times:(R^\times)^2]} \cdot |F|$.  
\item[(2)] If $G$ is of type ${^{3,6}}\text{D}_4$ or ${^2}\text{E}_6$ then $\tau(G) = \fc{|\ker(N^{(0)}[m])|}{|\ker(N^{(0)}/m)|} \cdot |F|$ (see Notation \ref{[m] and /m}). 
\end{itemize}
In both cases if $L$ is imaginary over $K$, then $\tau(G) = |F|$. 
\end{cor}

\begin{proof}
All groups under consideration are almost simple. 
When $G$ is adjoint of type ${^2}\text{D}_{2k}$ then $F$ is quasi-split, and when it is adjoint both of type ${^{3,6}}\text{D}_4$ or ${^2}\text{E}_6$   
then $F = \text{Res}^{(1)}_{L/K}(\mu_m)$ where $m$ is prime to $[L:K]$ (e.g., \cite[p.333]{PR}), thus $F$ is admissible.  
So the assertions $(1),(2)$ are just Theorem \ref{tau G 2} in which as $F_\iy$ splits, $|\wh{F_\iy}^{\fg_\iy}| = |F_\iy| = |F|$. 

As $C$ is projective, removing a single point $\iy$ from it implies that $\Oiy^\times = \BF_q^\times$ 
(an element of $\Oiy$ is regular at $\iy^{-1}$, thus its inverse is irregular there, hence not invertible in $\Oiy$,   
unless it is a unit). 
If $L$ is imaginary, then in particular $R = \BF_q[C'-\{\iy'\}]$ where $C'$ is a finite \'etale cover of $C$ and $\iy'$ is the unique point lying over $\iy$, 
thus still $R^\times = \BF_q^\times$ being finite, whence $|R^\times[2]| = [R^\times:(R^\times)^2]$. 
In the cases $F$ is not quasi-split the equality $R^\times = \Oiy^\times = \BF_q^\times$ means that $N^{(0)}$ is trivial, and we are done.  
\end{proof}

{\bf Acknowledgements:} 
I thank P.~Gille, B.~Kunyavski\u\i\ and U.~Vishne for valuable discussions concerning the topics of the present article. 
I would like also to thank the anonymous referee for a careful reading and many constructive remarks.

\end{document}